\documentclass{amsart}

\usepackage{latexsym,amssymb,amsmath, amsfonts,epsf}
\usepackage{tabularx}
\usepackage{amsthm, amsrefs}
\usepackage{amscd}

\newcommand{\R}{\mathbb{R}}

\newcommand{\la}{\lambda}

\newcommand{\F}{N^{-\lambda}}

\newtheorem{theorem}{Theorem}

\newtheorem{cor}{Corollary}

\theoremstyle{definition}

\address{Departamento de Matem\'aticas, Universidad Aut\'onoma de
Madrid, Madrid 28049, Spain}\email{keith.rogers@uam.es}
\address{University of California, Los Angeles, CA 90095--1555,
USA}\email{paco.villarroya@uv.es}

\author{Keith M. Rogers}
\author{Paco Villarroya}
\date{}
\title[Estimates for maximal operators associated to the wave
equation]{Sharp estimates for
  maximal operators\\associated to the wave equation}
\thanks{The first author is
supported by MEC project
 MTM2004-00678 and UAM-CM project CCG06-UAM/ESP-0286,
 and the second by MEC projects
MTM2004-21420 and MTM2005-08350.}

\begin{document}

\begin{abstract}
The wave equation, $\partial_{tt}u=\Delta u,$ in $\R^{n+1}$,
considered with initial data $u(x,0)=f\in H^s(\R^n)$ and
$u'(x,0)=0,$ has a solution which we denote by
$\frac{1}{2}(e^{it\sqrt{-\Delta}}f+e^{-it\sqrt{-\Delta}}f)$. We give
almost sharp conditions under which $\,\sup_{0<t<1}|e^{\pm
it\sqrt{-\Delta}}f|\,$ and $\,\sup_{t\in\R}|e^{\pm
it\sqrt{-\Delta}}f|$ are bounded from $H^s(\R^n)$ to
 $L^q(\R^n)$.
\end{abstract}

\maketitle

\section{Introduction}\label{intro}

The Schr\"odinger equation, $i\partial_tu+\Delta u=0,$ in
$\R^{n+1},$ with initial datum $f$ contained in a Sobolev space
$H^s(\R^n),$ has solution $e^{it\Delta}f$ which can be formally
written as
\begin{equation}\label{form}
e^{it\Delta}f(x)=\int \widehat{f}(\xi)e^{2\pi i(x\cdot\xi-2\pi
t|\xi|^2)}d\xi.
\end{equation}
The minimal regularity of $f$ under which $e^{it\Delta}f$ converges
almost everywhere to~$f$, as $t$ tends to zero, has been studied
extensively. By standard arguments, the problem reduces to the
minimal value of $s$ for which
\begin{equation}\label{two}
\|\sup_{0<t<1}|e^{it\Delta}f|\,\|_{L^q(\mathbb{B}^n)}\le C_{n,q,s}
\|f\|_{H^s(\R^n)}
\end{equation}
holds, where $\mathbb{B}^n$ is the unit ball in $\R^n$.

 In one spatial dimension, L.
Carleson \cite{ca} (see also \cite{keru}) showed that (\ref{two})
holds when $s\ge 1/4$, and B.E.J. Dahlberg and C.E. Kenig~\cite{da}
showed that this is sharp in the sense that it is not true when
$s<1/4$. In two spatial dimensions, significant contributions have
been made by J. Bourgain~\cite{bo0,bo1}, A. Moyua, A. Vargas and L.
Vega \cite{movave, movave2}, and T. Tao and Vargas \cite{ta0,ta1}.
The best known result is due to S. Lee \cite{le} who showed that
(\ref{two}) holds when $s>3/8$.
%(see also \cite{bo1} and \cite{mo}).
In higher dimensions, P. Sj\"olin \cite{sj} and L. Vega \cite{ve2}
independently showed that (\ref{two}) holds when $s>1/2$.

Replacing the unit ball $\mathbb{B}^n$ in (\ref{two}) by the whole
space $\R^n,$  there has also been significant interest (see
\cite{car}, \cite{co}, \cite{ro1}, \cite{rovave}, \cite{sj9},
\cite{ta0}, \cite{ta1}) in the global bounds
\begin{equation*}
\|\sup_{0<t<1}|e^{it\Delta}f|\,\|_{L^q(\R^n)}\le C_{n,q,s}
\|f\|_{H^s(\R^n)}
\end{equation*}
and
\begin{equation*}
\|\sup_{t\in\R}|e^{it\Delta}f|\,\|_{L^q(\R^n)}\le C_{n,q,s}
\|f\|_{H^s(\R^n)},
\end{equation*}
sometimes in connection with the well-posedness with certain initial
value problems~(see \cite{kepove2}). In one spatial dimension there
are almost sharp bounds (see \cite{kepove}, \cite{kepove2},
\cite{rovi}, \cite{sj0}, \cite{ve1}), but in higher dimensions the
problem remains open.

The wave equation, $\partial_{tt}u=\Delta u,$ in $\R^{n+1}$,
considered with initial data $u(\cdot,0)=f$ and $u'(\cdot,0)=0,$ has
solution which can be formally written as
$$\frac{1}{2}\left(e^{it\sqrt{-\Delta}}f+e^{-it\sqrt{-\Delta}}f\right)=\int
\widehat{f}(\xi)e^{2\pi i x\cdot\xi}\cos{(2\pi t|\xi|)}d\xi,
$$
where
\begin{equation}\label{form}
e^{\pm it\sqrt{-\Delta}}f(x)=\int \widehat{f}(\xi)e^{2\pi
i(x\cdot\xi\pm  t|\xi|)}d\xi.
\end{equation}

Mainly we will be concerned with
 the global bounds
\begin{equation}\label{th}
\|\sup_{0<t<1}|e^{\pm it\sqrt{-\Delta}}f|\,\|_{L^q(\R^n)}\le
C_{n,q,s} \|f\|_{H^s(\R^n)}
\end{equation}
and
\begin{equation}\label{fo}
\|\sup_{t\in\R}|e^{\pm it\sqrt{-\Delta}}f|\,\|_{L^q(\R^n)}\le
C_{n,q,s} \|f\|_{H^s(\R^n)}.
\end{equation}
We note that equation (\ref{fo}) is simply a mixed norm Strichartz
estimate.

 Everything that will follow is true for the solution to the wave equation
with initial derivative equal to zero, however, for notational
convenience, we will write things in terms of the one-sided
solutions $e^{\pm it\sqrt{-\Delta}}f$.

Let
$$
s_{n,q}=\max\left\{n\left(\frac{1}{2}-\frac{1}{q}\right),
\frac{n+1}{4}-\frac{n-1}{2q}\right\}\,\,\,\,\,\,\,\,\textrm{and}\,\,\,\,\,\,\,\,q_n=\frac{2(n+1)}{n-1}.$$
We will prove the following almost sharp theorems. The positive part
of Theorem~\ref{t1}, when $q=2$, is due to M. Cowling~\cite{co}.

\begin{theorem}\label{t1} If $q\in[2,\infty]$ and $s>s_{n,q},$ then
(\ref{th}) holds. If $q<2$ or $s<s_{n,q}$, then (\ref{th}) does not
hold.
\end{theorem}

\begin{theorem}\label{t2} If $q\in[q_n,\infty]$ and $s>n(\frac{1}{2}-\frac{1}{q}),$
then (\ref{fo}) holds. If $q<q_n$ or $s<n(\frac{1}{2}-\frac{1}{q})$,
then (\ref{fo}) does not hold.
\end{theorem}

We will also briefly consider the local bounds
\begin{equation}\label{th2}
\|\sup_{0<t<1}|e^{\pm it\sqrt{-\Delta}}f|\,\|_{L^q(\mathbb{B}^n)}\le
C_{n,q,s} \|f\|_{H^s(\R^n)}.
\end{equation}
and
\begin{equation}\label{fo2}
\|\sup_{t\in\R}|e^{\pm
it\sqrt{-\Delta}}f|\,\|_{L^q(\mathbb{B}^n)}\le C_{n,q,s}
\|f\|_{H^s(\R^n)}.
\end{equation}
That (\ref{th2}) and (\ref{fo2}) hold when $q\in [1,2]$ and $s>1/2$
is due to Vega~\cite{ve1,ve2}, and that this is not true when $s\le
1/2$ is due to B.G. Walther~\cite{wa}.

In the following theorem we prove that (\ref{th2}) does not hold
when $s<\frac{n+1}{4}-\frac{n-1}{2q}\,$ which is an improvement of
the fact that (\ref{th2}) does not hold when
$s<\frac{n}{4}-\frac{n-1}{2q}$, due to Sj\"olin \cite{sj5}.
\begin{theorem}
If $q\in[1,\infty]$ and $s>\max\{1/2,s_{n,q}\},$ then (\ref{th2})
and (\ref{fo2}) hold. If $s<\max\{1/2,s_{n,q}\}$, then (\ref{th2})
and (\ref{fo2}) do not hold.
\end{theorem}
\vspace{2em}
\[\begin{picture}(70,110)
\put(-20,0){\line(0,0){100}} \put(-20,0){\line(1,0){115}}
\put(-31,100){$s$} \put(47,24){\line(0,1){76}}
\put(-8,75){$(\ref{fo})$} \put(25,65){$(\ref{th})$}
\put(55,50){$(\ref{th2})\,\, \mathrm{ and }\,\, (\ref{fo2})$}
\put(15,40){\line(0,1){60}} \put(-20,75){\line(1,-1){35}}
 \put(15,40){\line(2,-1){32}}
\put(47,24){\line(1,0){31}}
\put(-32,23){$\frac{1}{2}$}\put(-32,72){$\frac{n}{2}$}\put(-38,38){$\frac{n}{n+1}$}
\put(2,-11){$\frac{n-1}{2(n+1)}$}\put(44,-11){$\frac{1}{2}$}\put(75,-11){$1$}\put(97,-11){$\frac{1}{q}$}
\put(-69,117){Region of boundedness for (\ref{th}), (\ref{fo}),
(\ref{th2}) and (\ref{fo2}).}
\end{picture}\]
\vspace{1em}

When $q=\infty$, there is a well known example (see for example
\cite{sj2}), that shows that $s>n/2$ is necessary for (\ref{th}),
(\ref{fo}), (\ref{th2}) and (\ref{fo2}) to hold. We also note that,
by the counterexample of Walther~\cite{wa}, $s>1/2$ is necessary for
(\ref{th}) to hold when $q=2$. We will not discuss these endpoint
cases further.

 Throughout, $C$ will denote an absolute constant whose value
may change from line to line.

\section{The positive results}

As usual, we define $\partial^\alpha_t$ by $
\widehat{\partial^\alpha_t
g}(\tau)=(2\pi|\tau|)^\alpha\widehat{g}(\tau),$ where $\alpha\ge 0.$
By the following theorem and Sobolev imbedding, we see that
(\ref{th}) and (\ref{fo}) hold when $q\ge q_n$  and
$s>n(\frac{1}{2}-\frac{1}{q}).$

\begin{theorem}\label{weak} Let $q\in[q_n,\infty)$
 and $s>\frac{n}{2}-\frac{n+1}{q}+\alpha$. Then there exists a
constant $C_{n,q,\alpha,s}$ such that $$
 \|\partial^\alpha_te^{\pm it \sqrt{-\Delta}}f\|_{L^q(\R^{n+1})} \le
 C_{n,q,\alpha,s}\|f\|_{H^s(\R^{n})}.
 $$
\end{theorem}

\begin{proof} First we observe that
 $\partial^\alpha_te^{\pm it \sqrt{-\Delta}}f=e^{\pm it
\sqrt{-\Delta}}f_\alpha,$ where
$\widehat{f_\alpha}(\xi)=(2\pi|\xi|)^\alpha\widehat{f}(\xi).$ Thus,
it will suffice to prove that
$$
\|e^{\pm it \sqrt{-\Delta}}f_\alpha\|_{L^q(\R^{n+1})} \le
 C_{n,q,\alpha,s}\|f_\alpha\|_{H^s(\R^{n})},
 $$
 where $q\ge q_n$ and $s>\frac{n}{2}-\frac{n+1}{q}$. By the standard
Littlewood--Paley arguments,
 it will suffice to show that
$$
\|e^{\pm it \sqrt{-\Delta}}g\|_{L^q(\R^{n+1})} \le
 C_{n,q}N^{n/2-(n+1)/q}\|g\|_{L^2(\R^n)},
 $$
 where $\textrm{supp } \widehat{g} \subset \{\xi\,:\, N/2\le|\xi|\le N\}$.

 Now by scaling, this is equivalent to
 $$
\|e^{\pm it \sqrt{-\Delta}}g\|_{L^q(\R^{n+1})} \le
 C_{n,q}\|g\|_{L^2(\R^n)},
 $$
  where $\,\textrm{supp } \widehat{g} \subset \{\xi\,:\, 1/2\le|\xi|\le 1\},$ which follows for
all
  $q\ge q_n$ by the Strichartz inequality~\cite{st}.
\end{proof}

It is tempting to try to increase the range of $q$ in the above
using bilinear restriction estimates on the cone as in \cite{ta1}.
Later we will see that this is not possible.

\begin{cor}\label{cord}
If $q\in[q_n,\infty)$
 and $s>n(\frac{1}{2}-\frac{1}{q})$, then (\ref{th}) and (\ref{fo}) hold.
\end{cor}

The following theorem is a corollary of a more general result due to
Cowling~\cite{co}.

\begin{theorem}
If q=2 and $s>1/2$ then (\ref{th}) holds.
\end{theorem}

Considering $H^s$ to be a weighted $L^2$ space, we interpolate
between Corollary~\ref{cord} with $q=q_n$, and the previous theorem
to get the following corollary.

\begin{cor}
If $q\in [2,q_n]$ and $s>\frac{n+1}{4}-\frac{n-1}{2q},$ then
(\ref{th}) holds.
\end{cor}

\section{The negative results}

%We start with the more trivial necessary conditions. That these
%conditions are necessary for (\ref{fo}) to hold, can actually been
%seen just by scaling.
\begin{theorem}\label{thr}
If (\ref{th}) holds, then $q\in [2,\infty]$ and $s\ge s_{n,q}.$ If
(\ref{fo}) holds, then $q\in [q_n,\infty]$ and $s\ge
n(\frac{1}{2}-\frac{1}{q}).$ If (\ref{th2}) or (\ref{fo2}) hold then
$s\ge\max\{1/2,s_{n,q}\}$.
\end{theorem}
%\begin{theorem}\label{1dt}
%If (\ref{th}) or (\ref{fo}) holds, then $q\in[2,\infty]$ and $s\ge
%n\left(\frac{1}{2}-\frac{1}{q}\right)$.
%\end{theorem}
\begin{proof} By a change of variables, it will suffice to consider $e^{-
it\sqrt{-\Delta}}f$. First we obtain necessary conditions for
$\,\sup_{t\in\R}|e^{- it\sqrt{-\Delta}}f|\,$, and then add the
condition $t\in(0,1)$, to obtain necessary conditions for
$\,\sup_{0<t<1}|e^{- it\sqrt{-\Delta}}f|$.

Let $A$ be a set contained in the ball $B(0,N)$, where $N\gg1,$ and
define $f_A$ by $f_A=\widehat{\chi}_A^{}$. Recall that
$$
\sup_{t\in\R}|e^{-it\sqrt{-\Delta}}f_A|=\sup_{t\in\R}\left|\int_Ae^{2\pi
i(x\cdot\xi- t|\xi|)}d\xi\right|.
$$
 The basic idea that we exploit, is to choose sets $A$ and $E$ for
which a time $t(x)$ can be chosen, so that the phase
$2\pi(x\cdot\xi- t(x)|\xi|)$ is almost zero for all $\xi\in A$ and
$x\in E$. Then, as $ \cos(2\pi (x\cdot\xi- t(x)|\xi|))\ge C,$ we see
that
$$
\|\sup_{t\in\R}|e^{-it\sqrt{-\Delta}}f_A|\,\|_{L^q(\R^n)}\ge
\left(\int_{E} (C|A|)^q\right)^{1/q}\ge C|A||E|^{1/q}.$$ On the
other hand,
$$
\|f_A\|_{H^s(\R^n)}\le\left(\int_A (1+|\xi|)^{2s}\right)^{1/2}\le
|A|^{1/2}(1+N)^s,
$$
so that, as
$\|\sup_{t\in\R}|e^{-it\sqrt{-\Delta}}f_A|\|_{L^q(\R^n)}\le C
\|f_A\|_{H^s(\R^n)}$, we have
\begin{equation}\label{fundy}
|A|^{1/2}|E|^{1/q}\le CN^s
\end{equation}
for all $N\gg 1.$

 When $n=1,$ we let $t(x)=x,$ so that the phase
is equal to zero for all $\xi\in [0,N],$ and $x\in\R$. Thus,
substituting $|E|=|\R|$ in (\ref{fundy}), we see there can be no
bound for $q<\infty.$ When $q=\infty$, substituting $|A|=N$ into
(\ref{fundy}), we see that $s\ge 1/2$, and we have the necessary
conditions for (\ref{fo}). Substituting $|A|=N$ and $E=[0,1]$ into
(\ref{fundy}), we see that $s\ge 1/2$, and we have the necessary
conditions for (\ref{fo2}).

Considering $\,\sup_{0<t<1}|e^{-it\sqrt{-\Delta}}f_A|\,$, we have
the added constraint that we must choose $t(x)$ in the interval
$(0,1)$. Choosing $t(x)=x$ again, and $E=(0,1)$, we see that $s\ge
1/2.$ We note that this is a necessary condition for (\ref{th2}) as
well as (\ref{th}).  That (\ref{th}) does not hold when $q<2$,
follows from an example in \cite{sj0}.

When $n\ge2$, define $A$ by
$$
A=\left\{\xi\in\R^n \,:\, |\theta_{\xi,e_n}|<\frac{\F}{10} \textrm{
and } |\xi|< N\right\},
$$ where $N\gg1,$ $\la\in[0,\infty)$ and $\theta_{\xi,e_n}$ denotes
the angle between $\xi$ and the standard basis vector $e_n$.
Similarly we define $E$ by
$$E=\left\{x\in\R^n \,:\, |\theta_{x,e_n}|<\F
\textrm{ and } |x|< N^{2\la-1}\right\},$$ and let $t(x)=|x|$. Given
that $$|\cos \theta_{\xi,x}-1|\le\left(\frac{\F}{5}\right)^2,$$ we
have
\begin{align*}
|2\pi(x\cdot\xi- t(x)|\xi|)|&= 2\pi|\xi||x||\cos
\theta_{\xi,x}-1|\\
&\le 2\pi N N^{2\la-1}\left(\frac{\F}{5}\right)^2\le\frac{2\pi}{25},
\end{align*}
so that the phase is always close to zero. Now as
$$|A|\ge C_n N (N^{1-\la})^{n-1} \,\,\,\,\,\,\,\,\,\,\textrm{ and
}\,\,\,\,\,\,\,\,\,\, |E|\ge C_n N^{2\la-1} (N^{\la-1})^{n-1},$$ we
see from (\ref{fundy}), that
$$
N^s\ge C_n N^{\frac{n-\la(n-1)}{2}}N^\frac{(n+1)\la-n}{q}
$$
for all $N\gg 1$, so that
$$
s\ge
n\left(\frac{1}{2}-\frac{1}{q}\right)-\la\left(\frac{n-1}{2}-\frac{n+1}{q}\right).
$$

Letting $\la=0$, we see that $s\ge n(\frac{1}{2}-\frac{1}{q})$. When
$q<q_n$, we have $\frac{n-1}{2}-\frac{n+1}{q}<0$, so that we can let
$\la\to\infty$ to get a contradiction for all $s$. This completes
the sufficient conditions for $(\ref{fo})$.

Considering $\,\sup_{0<t<1}|e^{-it\sqrt{-\Delta}}f_A|$, we have the
added condition that $t(x)<1.$ This is fulfilled if $\la\le1/2,$ so
that $|x|< 1$. Letting $\la=0,$  we have $s\ge
n(\frac{1}{2}-\frac{1}{q})$ as before, and letting $\la=1/2$, we get
$s\ge \frac{n+1}{4}-\frac{n-1}{2q}$. We note that these are also
necessary conditions for the local bounds.

It remains to prove that $q\ge 2$ is necessary for the global
boundedness of $\sup_{0<t<1}|e^{-it\sqrt{-\Delta}}f|$, and that
$s\ge 1/2$ is necessary for the local bounds. These will require
separate constructions.

For the global bound, we consider $A$ as defined before with $E$
defined by
$$
E=\left\{x\in\R^n\,:\, |\theta_{x,e}|\le N^{-\la} \textrm{ for some
} e\in\textrm{span}\{e_1,\ldots,e_{n-1}\}, \textrm{ and } |x|<
\frac{N^{\la-1}}{10}\right\},
$$
where $\la\in [0,\infty)$, and we let $t(x)=0$. Then
\begin{align*}
|2\pi(x\cdot\xi- t(x)|\xi|)|&= 2\pi|\xi||x||\cos
\theta_{\xi,x}|\\
&=2\pi|\xi||x||\sin(\pi/2- \theta_{\xi,x})|\\ &\le 2\pi N
\frac{N^{\la-1}}{10}2N^{-\la}\le\frac{4\pi}{10},
\end{align*}
so that the phase is always close to zero. Now as
$$|A|\ge C_n N (N^{1-\la})^{n-1} \,\,\,\,\,\,\,\,\,\,\textrm{ and
}\,\,\,\,\,\,\,\,\,\, |E|\ge C_n N^{-1} (N^{\la-1})^{n-1},$$ we see
from (\ref{fundy}), that
$$
N^s\ge C_n N^{\frac{n-\la(n-1)}{2}}N^\frac{\la(n-1)-n}{q},
$$
so that
$$
s\ge
n\left(\frac{1}{2}-\frac{1}{q}\right)-\la\left(\frac{n-1}{2}-\frac{n-1}{q}\right).
$$
We see that when $q<2,$ we can let $\la\to\infty$ to get a
contradiction for all~$s$.

Finally, for the local bounds, we define $A$ and $E$ by
$$
A=\left\{\xi\in\R^n \,:\, |\theta_{\xi,e_n}|<1/N \textrm{ and }
|\xi|< N\right\}, $$
$$E=\left\{x\in\R^n \,:\, |\theta_{x,e_n}|<1/100
\textrm{ and } |x|< 1\right\},$$ and let $t(x)=|x|\cos
\theta_{x,e_n}$. Now using the inequality $|\cos x-\cos y|\leq
|x^2-y^2|$ we have
\begin{align*}
|2\pi(x\cdot\xi- t(x)|\xi|)|&= 2\pi|\xi||x||\cos \theta_{x,\xi}-\cos
\theta_{x,e_n}|\\
&\le 2\pi N\left|\theta_{x,\xi}^2-\theta_{x,e_n}^2\right|\\
&= 2\pi
N\left|\theta_{x,e_n}-\theta_{x,\xi}\right|\left|\theta_{x,e_n}+\theta_{x,\xi}\right|\\
&\le2\pi N\frac{1}{N}(\frac{1}{100}+\frac{1}{N})\le\frac{1}{3},
\end{align*}
so that the phase is always close to zero. Now as
$$|A|\ge C_n N  \,\,\,\,\,\,\,\,\,\,\textrm{ and }\,\,\,\,\,\,\,\,\,\,
|E|\ge C_n,$$
 we see from (\ref{fundy}), that
$$
N^s\ge C_n N^{1/2}
$$
for all $N\gg1$, so that $s\ge1/2$ and this completes the necessary
conditions for local boundedness.
\end{proof}
\vspace{1em}

Thanks to the referee for bringing an important reference to our
attention.

\end{document}